\documentclass{amsart}

\usepackage{amsfonts}
\usepackage{amsmath}
\usepackage{amssymb}
\usepackage{amsthm}

\makeatletter
\renewcommand\@seccntformat[1]{\csname the#1\endcsname.\quad}

\newtheorem{theorem}{Theorem}[section]

\newtheorem{corollary}[theorem]{Corollary}

\newtheorem{lemma}[theorem]{Lemma}

\newtheorem{remark}[theorem]{Remark}

\newcommand{\norm}[3]{\ensuremath{\left\Vert#1\right\Vert_{#2}^{#3}}}
\newcommand{\abs}[3]{\ensuremath{\left\vert#1\right\vert_{#2}^{#3}}}
\DeclareMathOperator{\supp}{supp}

\DeclareMathOperator{\sg}{sgn}

\DeclareMathOperator{\card}{card}
\DeclareMathOperator{\id}{Id}

\begin{document}

\author[M. Ciesielski \& G. Lewicki]{Maciej Ciesielski$^{1*}$ and Grzegorz Lewicki}

\title[Sequence Lorentz spaces and their geometric structure]{Sequence Lorentz spaces and their geometric structure}

\begin{abstract}
	This article is dedicated to geometric structure of the Lorentz and Marcinkiewicz spaces in case of the pure atomic measure. We study complete criteria for order continuity, the Fatou property, strict monotonicity and strict convexity in the sequence Lorentz spaces $\gamma_{p,w}$. Next, we present a full characterization of extreme points of the unit ball in the sequence Lorentz space $\gamma_{1,w}$. We also establish a complete description with an isometry of the dual and predual spaces of the sequence Lorentz spaces $\gamma_{1,w}$ written in terms of the Marcinkiewicz spaces. Finally, we show a fundamental application of geometric structure of $\gamma_{1,w}$ to one-complemented subspaces of $\gamma_{1,w}$.	
\end{abstract}

\maketitle

\bigskip\ 

{\small \underline{2000 Mathematics Subjects Classification: 46E30, 46B20, 46B28.}\hspace{1.5cm}\
\ \ \quad\ \quad  }\smallskip\ 

{\small \underline{Key Words and Phrases:}\hspace{0.15in} Lorentz and Marcinkiewicz spaces, strict monotonicity, strict convexity, order continuity, extreme point, existence set, one-complemented subspace.}

\bigskip\ \ 

\section{Introduction}

Geometric structures with application of the Lorentz spaces $\Gamma_{p,w}$ and Marcinkiewicz spaces $M_\phi$ in case of the non-atomic measure have been investigated extensively by many authors \cite{Cies-geom,CKKP,CiesKamPluc,KMGam,KamMal}. In contrast to the non-atomic case there are only few papers concerning geometric structure of sequence Lorentz and Marcinkiewicz spaces. The first crucial paper devoted to the Marcinkiewicz spaces appeared in 2004 \cite{KamLee}, where authors have studied the biduals and order continuous ideals of the Marcinkiewicz spaces for the pure atomic measure. The next significant paper was published in 2009 \cite{KamLeeLew}, in which there has been investigated, among others, strict monotonicity, smooth points and extreme points with application to one-complemented subspaces. For other results concerning the issue devoted to one-complemented subspaces please see a.g. \cite{DaEn,JamKamLew,KamLew}.

The purpose of this article is to explore geometric properties of the sequence Lorentz spaces $\gamma_{p,w}$ and its dual and predual spaces. It is worth mentioning that we present an application of geometric properties to a characterization of one-complemented subspaces in the Lorentz spaces $\gamma_{p,w}$ in case of the pure atomic measure. It is necessary to mention that a characterization of geometric structure of the sequence Lorentz and Marcinkiewicz spaces does not follow immediately as a consequence of well known results from the case of non-atomic measure in general. 
 
The paper is organized as follows. In section 2, we present the needed terminology. In section 3, we show an auxiliary result devoted to a relationship between the global convergence in measure of a sequence $(x_n)\subset\ell^0$ and the pointwise convergence of its sequence of decreasing rearrangements $(x_n^*)$. In case of the pure atomic measure, we also establish a correspondence between an identity of signs of the values for two different sequences in $\ell^0$ and an additivity of the decreasing rearrangement operation for these sequences. Section 4 is devoted to an investigation of geometric structure of sequence Lorentz spaces $\gamma_{p,w}$. Namely, we focus on complete criteria for order continuity and the Fatou property in Lorentz spaces for the pure atomic measure. Next, we present a characterization of strict monotonicity and strict convexity of $\gamma_{p,w}$ written in terms of the weight sequence $w$. In spirit of the previous result, we describe an equivalent condition for extreme points of the unit ball in the sequence Lorentz space $\gamma_{1,w}$. In section 5, we solve the essential problem showing a full description of the dual and predual spaces of the sequence Lorentz space $\gamma_{1,w}$. First, we answer a crucial question under which condition does an isometric isomorphism exist between the dual space of the sequence Lorentz space $\gamma_{1,w}$ and the sequence Marcinkiewicz space $m_\phi$. Next, we discuss complete criteria which guarantee that the predual space of the sequence Lorentz space $\gamma_{1,w}$ coincides with the sequence Marcinkiewicz space $m_\phi^0$. Additionally, we investigate necessary condition for the isometry between the predual of $\gamma_{1,w}$ and the Marcinkiewicz space $m_\phi^0$. In section 5, we present an application of geometric properties of the sequence Lorentz space $\gamma_{1,w}$ to a characterization of one-complemented subspaces. Namely, using an isometry between the classical Lorentz space $d_{1,w}$ and the Lorentz space $\gamma_{1,w}$, we prove that there exists norm one projection on any nontrivial existence subspace of $\gamma_{1,w}$. Additionally, by the previous investigation and in view of \cite{KamLeeLew}, we establish a full characterization of smooth points in the sequence Lorentz space $\gamma_{1,w}$ and its predual and dual spaces. Finally, we study an equivalent condition for an extreme points in the dual space of the sequence Lorentz space $\gamma_{1,w}$.

\section{Preliminaries}

Let $\mathbb{R}$, $\mathbb{R}^+$ and $\mathbb{N}$ be the sets of reals, nonnegative reals and positive integers, respectively.  A mapping $\phi:\mathbb{N}\rightarrow\mathbb{R}^+$ is said to be \textit{quasiconcave} if $\phi(t)$ is increasing and $\phi(t)/t$ is decreasing on $\mathbb{N}$ and also $\phi(n)>0$ for all $n\in\mathbb{N}$. We denote by $\ell^{0}$ the set of all real sequences, and by $S_X$ (resp. $B_X)$ the unit sphere (resp. the closed unit ball) in a Banach space $(X,\norm{\cdot}{X}{})$. Let us denote by $(e_i)_{i=1}^\infty$ a standard basis in $\mathbb{R}^\infty$. A sequence quasi-Banach lattice $(E,\Vert \cdot \Vert _{E})$ is said to be a \textit{quasi-Banach sequence space} (or a \textit{quasi-K\"othe sequence space}) if it is a sequence sublattice of $\ell^{0}$ and holds the following conditions
\begin{itemize}
\item[$(1)$] If $x\in\ell^0$, $y\in E$ and $|x|\leq|y|$, then $x\in E$ and $\|x\|_E\leq\|y\|_E$.
\item[$(2)$] There exists a strictly positive $x\in E$.
\end{itemize}
For simplicity let us use the short symbol $E^{+}={\{x \in E:x \ge 0\}}$. 
% OC, Fatou property, UR
An element $x\in E$ is called a \textit{point of order continuity}, shortly $x\in{E_a}$, if for any
sequence $(x_{n})\subset{}E^+$ such that $x_{n}\leq \left\vert x\right\vert 
$ and $x_{n}\rightarrow 0$ pointwise we have $\left\Vert x_{n}\right\Vert
_{E}\rightarrow 0.$ A quasi-Banach sequence space $E$ is said to be \textit{order continuous}, shortly $E\in \left( OC\right)$, if any element $x\in{}E$ is a point of order continuity. A space $E$ is said to be \textit{reflexive} if $E$ and its associate space $E'$ are order continuous. Given a quasi-Banach sequence space $E$ is said to have the \textit{Fatou property} if for all $\left( x_{n}\right)\subset{}E^+$, $\sup_{n\in \mathbb{N}}\Vert x_{n}\Vert
_{E}<\infty$ and $x_{n}\uparrow x\in\ell^{0}$, then $x\in E$ and $\Vert x_{n}\Vert _{E}\uparrow\Vert x\Vert
_{E}$ (see \cite{LinTza,BS}). We say that $E$ is \textit{strictly monotone} if for any $x,y\in{E^+}$ such that $x\leq{y}$ and $x\neq{y}$ we have $\norm{x}{E}{}<\norm{y}{E}{}$.

% strict rotundity
Let $(X,\norm{\cdot}{X}{})$ be a Banach space. Recall that $x\in{S_X}$ is an \textit{extreme point} of $B_X$ if for any $y,z\in{S_X}$ such that $x=(y+z)/2$ we have $x=y=z$. A Banach space $X$ is called \textit{rotund} or \textit{strictly convex} if any $x\in{}S_X$ is an extreme point of $B_X$. An element $x\in{X}$ is called a \textit{smooth point} of $X$ if there exists a unique linear bounded functional $f\in{S_{X^*}}$ such that $f(x)=\norm{x}{X}{}$. 

%symmetric spaces and properties
The \textit{distribution} for any sequence $x\in\ell^{0}$ is defined by 
\begin{equation*}
d_{x}(\lambda) =\card\left\{k\in\mathbb{N}:\left\vert x\left(k\right) \right\vert >\lambda \right\},\qquad\lambda \geq 0.
\end{equation*}
For any sequence $x\in\ell^{0}$ its \textit{decreasing rearrangement} is given by 
\begin{equation*}
x^{*}\left(n\right) =\inf \left\{ \lambda\geq 0:d_{x}\left( \lambda
\right)\leq n-1\right\}, \text{ \ \ }\quad{n\in\mathbb{N}}.
\end{equation*}
In this article we use the notation $x^{*}(\infty)=\lim_{n\rightarrow\infty}x^{*}(n)$. For any sequence $x\in\ell^{0}$ we denote the \textit{maximal sequence} of $x^{\ast }$ by 
\begin{equation*}
x^{\ast \ast }(n)=\frac{1}{n}\sum_{i=1}^{n}x^{*}(i).
\end{equation*}
It is easy to notice that for any point $x\in\ell^{0}$, $x^{\ast }\leq x^{\ast \ast },$ $x^{\ast \ast }$ is decreasing, continuous and subadditive. For more details of $d_{x}$, $x^{\ast }$ and $x^{\ast \ast }$ see \cite{BS,KPS}. 

We say that two sequences $x,y\in{\ell^0}$ are \textit{equimeasurable}, shortly $x\sim y$, if $d_x=d_y$. A quasi-Banach sequence space $(E,\norm{\cdot}{E}{})$ is called \textit{symmetric} or \textit{rearrangement invariant} (r.i. for short) if whenever $x\in\ell^{0}$ and $y\in E$ such that $x \sim y,$ then $x\in E$ and $\Vert x\Vert_{E}=\Vert y\Vert _{E}$. The \textit{fundamental sequence} $\phi_E$ of a symmetric space $E$ we define as follows $\phi_{E}(n)=\Vert\chi_{\{i\in\mathbb{N}:i\leq n\}}\Vert_{E}$ for any $n\in\mathbb{N}$ (see \cite{BS}). Let $0<p<\infty$ and $w=(w(n))_{n\in\mathbb{N}}$ be a nonnegative real sequence and let for any $n\in\mathbb{N}$
\begin{equation*}
W(n)=\sum_{i=1}^n{w(i)}\quad\textnormal{and}\quad{W}_p(n)=n^p\sum_{i=n+1}^\infty\frac{w(i)}{i^p}<\infty.
\end{equation*} 
For short notation the sequence $w$ is called a nonnegative weight sequence. In the whole paper, unless we say otherwise we suppose that $w$ a nonnegative weight sequence is nontrivial, i.e. there is $n\in\mathbb{N}$ such that $w(n)>0$. Now, we recall the sequence Lorentz space $d_{1,w}$ which is a subspace of $\ell^0$ such that for any sequence $x=(x(n))_{n\in\mathbb{N}}\in{d_{1,w}}$ we have
\begin{equation*}
\norm{x}{d_{1,w}}{}=\sum_{i=1}^\infty{x}^{*}(n)w(n)<\infty.
\end{equation*}
It is well known that the Lorentz space $d_{1,w}$ is a symmetric space with the Fatou property (see \cite{KamMal}). The sequence Lorentz space $\gamma_{p,w}$ is a collection of all real sequences $x=(x(n))_{n\in\mathbb{N}}$ such that 
\begin{equation*}
\norm{x}{\gamma_{p,w}}{}=\left(\sum_{i=1}^\infty({x}^{**}(n))^pw(n)\right)^{1/p}<\infty.
\end{equation*}
Let us notice that for any nonnegative sequence $w=(w(n))_{n\in\mathbb{N}}$ the sequence Lorentz space $\gamma_{p,w}$ is a r.i. (quasi-)Banach sequence space equipped with the (quasi-)norm $\norm{\cdot}{\gamma_{p,w}}{}$. It is easy to observe that the fundamental sequence of the Lorentz space $\gamma_{p,w}$ is given by $\phi_{\gamma_{p,w}}(n)=\norm{\chi_{\{i\leq{n},i\in\mathbb{N}\}}}{\gamma_{p,w}}{}=(W(n)+W_p(n))^{1/p}$ for every $n\in\mathbb{N}$. Let $\phi$ be a quasiconcave sequence. The Marcinkiewicz space $m_{\phi}$ and (resp. $m_\phi^0$) consists of all real sequences $x=(x(n))_{n\in\mathbb{N}}$ such that 
\begin{equation*}
\norm{x}{m_\phi}{}=\sup_{n\in\mathbb{N}}\left\{x^{**}(n)\phi(n)\right\}<\infty\quad\left(\textnormal{resp.}\quad{m_\phi^0\subset}m_\phi\quad\textnormal{and}\quad\lim_{n\rightarrow\infty}x^{**}(n)\phi(n)=0\right).
\end{equation*}
Recall that $m_\phi$ and $m_\phi^0$ are symmetric spaces equipped with the norm $\norm{\cdot}{m_\phi}{}$ (for more details see \cite{KamLee}).

\section{properties of decreasing rearrangement for a pure atomic measure}

In this section, first we present an auxiliary lemma devoted to a correspondence between the global convergence in measure on $\mathbb{N}$ of an arbitrary sequence of elements in $\ell^0$ to an element in $\ell^0$ and the pointwise convergence of their decreasing rearrangements. Although the similar result emerges in case of the non-atomic measure space (see \cite{KPS}), the proof of it is not valid in case of the pure atomic measure space. It is worth mentioning that in the pure atomic measure space the proof of the wanted result is quite long and requires new techniques. 

\begin{lemma}\label{lem:properties}
	Let $x_m,x\in\ell^0$ for all $m\in\mathbb{N}$. If $x_m$ converges to $x$ globally in measure, then $x_m^*$ converges to $x^*$ on $\mathbb{N}$.
\end{lemma}

\begin{proof}
	Let $(x_m)\subset\ell^0$, $x\in\ell^0$ be such that $x_m\rightarrow{x}$ globally in measure. Since for any $\epsilon>0$ and $m\in\mathbb{N}$ we have
	\begin{equation*}
	\card\{n\in\mathbb{N}:\abs{x_m(n)-x(n)}{}{}>\epsilon\}\geq\card\{n\in\mathbb{N}:\abs{|x_m(n)|-|x(n)|}{}{}>\epsilon\},
	\end{equation*}
	without loss of generality we may assume that $x\geq{0}$ and $x_m\geq{0}$ for all $n\in\mathbb{N}$. Let $B=\{b_i\}$ be a set of all values for a function $x:\mathbb{N}\rightarrow\mathbb{R}^+$. Define for any $i\in\{1,\dots,\card(B)\}$,
	\begin{equation*}
	N_i=\{n\in\mathbb{N}:x(n)=b_i\},\quad\textnormal{and}\quad c_i=\sum_{j=1}^{i}\card(N_j),\quad{c_0}=0.
 	\end{equation*}
	Without loss of generality we may assume that $(b_i)$ is strictly decreasing. Now we present the proof in three cases.\\
	\textit{Case $1.$} Suppose that $\card(N_1)=\infty$. Then, it is easy to see that $x^*(n)=b_1\chi_{\mathbb{N}}$. If $b_1=0$ then for all $m\geq M_{\delta_1}$ we have
	$$d_{x_m}(\delta_1)=\card\{n\in\mathbb{N}:|x_m(n)|>\delta_1\}<1.$$
	Hence, since $d_{x_m^*}(\delta_1)=d_{x_m}(\delta_1)$ for every $m\geq{M_{\delta_1}}$, we get $x_m^*\rightarrow{0}$ globally in measure, whence we infer that $x_m^*\rightarrow{0}$ pointwise. In case when $B=\{b_1\}$ then we take $b_2=0$. Denote $\delta_1=(b_1-b_2)/4$. Since $x_m\rightarrow{x}$ globally in measure, there exists $M_{\delta_1}\in\mathbb{N}$ such that for all $m\geq M_{\delta_1}$, 
	\begin{equation}\label{equ:1:converg}
	\card\{n\in\mathbb{N}:\abs{x_m(n)-x(n)}{}{}>\delta_1\}<1.
	\end{equation}
	Now, we claim that for any $n\in\mathbb{N}$, $x_m^*(n)\rightarrow{x^*(n)}$. Indeed, by \eqref{equ:1:converg} we conclude that for any $m\geq{M_{\delta_1}}$ and $n\in\mathbb{N}$,
	\begin{equation*}
	\abs{x(n)-x_m(n)}{}{}\leq\delta_1.
	\end{equation*}
	If $\card(\mathbb{N}\setminus{N_1})=0$, then we are done. Otherwise, for any $n\in{N_1}$ and $k\in\mathbb{N}\setminus{N_1}$ we observe that 
	\begin{equation*}
	x_m(n)\geq{x(n)-\delta_1}=b_1+\frac{3(b_1-b_2)}{4}=b_1+3\delta_1> x(k)+3\delta_1\geq x_m(k)+2\delta_1
	\end{equation*}
	for all $m\geq M_{\delta_1}$. Consequently, for every $m\geq M_{\delta_1}$ we obtain $x_m^*=\left(x_m\chi_{N_1}\right)^*$ and also $\abs{b_1-x_m(n)}{}{}\leq\delta_1$ for each $n\in{N_1}$. Therefore, for all $m\geq{M_{\delta_1}}$ and $n\in\mathbb{N}$ it is easy to notice that $$\delta_1\geq\abs{b_1-x_m^*(n)}{}{}=\abs{x^*(n)-x_m^*(n)}{}{}.$$
	\textit{Case $2.$} Assume that there exists $b_{j_0}\in{B}\setminus\{0\}$ such that $\card(N_{j_0})=\infty$ and $0<\card(N_j)<\infty$ for any $j\in\{1,\dots,j_0-1\}$. 
	Then, we have
	\begin{equation}\label{equ:2:converg}
	x^*(n)=\left(\sum_{j=1}^{j_0}b_j\chi_{N_j}\right)^*(n)=\sum_{j=1}^{j_0}b_j\chi_{\{i\in\mathbb{N}:c_{j-1}+1\leq{i}\leq{c_j}\}}(n).
	\end{equation}
	In case when $\card{(B)}={j_0}$ then we assume that $b_{j_0+1}=0$. Denote for any $i\in\{1,\dots,\card(B)\}$,
	\begin{equation*}
	\delta_i=\frac{b_i-b_{i+1}}{4}\qquad\textnormal{and}\qquad\delta=\min_{1\leq i\leq{j_0}}\{\delta_i\}.
	\end{equation*}
	Since $x_m\rightarrow{x}$ globally in measure, there exists $M_{\delta}\in\mathbb{N}$ such that for all $m\geq M_{\delta}$, 
	\begin{equation*}
	\card\{n\in\mathbb{N}:\abs{x_m(n)-x(n)}{}{}>\delta\}<1.
	\end{equation*}
	Therefore, for any $m\geq{M_\delta}$ and $n_i\in{N_i}$ where $1\leq i\leq j_0$ we have
	\begin{equation}\label{equ:3:converg}
	\delta\geq\abs{x(n_i)-x_m(n_i)}{}{}=\abs{b_i-x_m(n_i)}{}{}.
	\end{equation}
	Hence, for all $m\geq{M_\delta}$ and $n_i\in{N_i}$ where $1\leq i\leq j_0-1$ we easily observe
	\begin{equation*}
	x_m(n_i)=b_i-\delta\geq b_{i+1}+3\delta\geq x_m(n_{i+1})+2\delta. 
	\end{equation*}
	In consequence, by \eqref{equ:3:converg} we get for every $m\geq{M_\delta}$ and $n\in\mathbb{N}$,
	\begin{equation}\label{equ:4:converg}
	x_m^*(n)=\left(\sum_{j=1}^{j_0}x_m\chi_{N_j}\right)^*(n)=\sum_{j=1}^{j_0}\left(x_m\chi_{N_j}\right)^*(n-c_{j-1})\chi_{\{i\in\mathbb{N}:c_{j-1}+1\leq{i}\leq{c_j}\}}(n).
	\end{equation}
	Clearly, there exists $\sigma:\mathbb{N}\rightarrow\bigcup_{j=1}^{j_0}N_j$ a permutation such that $x^*(n)=x(\sigma(n))$ for all $n\in\mathbb{N}$. Thus, for any $n\in\mathbb{N}$ there exists $j\in\{1,\dots,j_0\}$ such that $\sigma(n)\in{N_j}$ and by \eqref{equ:3:converg} we obtain
	\begin{equation*}
		\delta\geq|x_m(\sigma(n))-x(\sigma(n))|=|x_m(\sigma(n))-b_j|=|(x_m\chi_{N_j})^*(n-c_{j-1})-b_j|
	\end{equation*}
	for all $m\geq{M_{\delta}}$. Therefore, by \eqref{equ:2:converg} and \eqref{equ:4:converg} we infer that
	\begin{align*}
	x_m^*(n)=&\sum_{j=1}^{j_0}\left(x_m\chi_{N_j}\right)^*(n-c_{j-1})\chi_{\{i\in\mathbb{N}:c_{j-1}+1\leq{i}\leq{c_j}\}}(n)\\
	&\rightarrow\sum_{j=1}^{j_0}b_j\chi_{\{i\in\mathbb{N}:c_{j-1}+1\leq{i}\leq{c_j}\}}(n)=x^*(n).
	\end{align*}
	\textit{Case $3.$} Suppose that for any $b_j\in{B}\setminus\{0\}$ we have $\card(N_j)<\infty$. If $\card(B)<\infty$ then without loss of generality we may assume that $j_0=\card(B)$ and $b_{j_0}=0$. Next, letting for any $i\in\{1,\dots,j_0-1\}$,
	\begin{equation*}
	\delta_i=\frac{b_i-b_{i+1}}{4}\qquad\textnormal{and}\qquad\delta=\min_{1\leq i\leq{j_0-1}}\{\delta_i\},
	\end{equation*}	
	and proceeding analogously as in case $2$ we may show that $x_m^*\rightarrow{x^*}$ on $\mathbb{N}$, in case when $\card(B)<\infty$. Now, assume that $\card(B)=\infty$. Then, since $(b_j)$ is strictly decreasing and bounded we conclude 
	\begin{equation*}
	\lim_{j\rightarrow\infty}b_j=b\geq{0}.
	\end{equation*}
	First, let us consider that $b=0$. Let $\epsilon>0$. Then, there exists $j_0\in\mathbb{N}$ such that for all $j\geq{j_0}$ we have 
	\begin{equation}\label{equ:5:converg}
	0<b_j<\frac{\epsilon}{4}\qquad\textnormal{and}\qquad{b_{j_0-1}\geq\frac{\epsilon}{4}}.
	\end{equation}
	Define for any $i\in\{1,\dots,j_0\}$,
	\begin{equation*}
	\delta_i=\frac{b_i-b_{i+1}}{4}\qquad\textnormal{and}\qquad\delta=\min\left\{\frac{\epsilon/4-b_{j_0}}{4},\min_{1\leq i\leq{j_0}}\{\delta_i\}\right\}.
	\end{equation*}
	Similarly as in case $2$ there is $M_\delta\in\mathbb{N}$ such that for all $m\geq{M_\delta}$, $n\in\mathbb{N}$ and $k\in\bigcup_{j\geq j_0}{N_j}$ we get
	\begin{equation}\label{equ:6:converg}
	\abs{x_m(n)-x(n)}{}{}\leq\delta\quad\textnormal{and}\quad{}x_m(k)\leq{\delta}+x(k)\leq\delta+b_{j_0}<\delta+\frac{\epsilon}{4}<\frac{\epsilon}{2}.
	\end{equation}
	Moreover, we may observe that 
	\begin{equation*}
	x_m(n_i)\geq x_m(n_{i+1})+2\delta
	\end{equation*}
	for every $m\geq{M_\delta}$ and $n_i\in{N_i}$ where $i\in\{1,\dots,j_0-1\}$. Next, assuming that $\sigma:\mathbb{N}\rightarrow\bigcup_{j=1}^{\infty}N_j$ is a permutation such that $x^*(n)=x(\sigma(n))$ for all $n\in\mathbb{N}$, then for any $n\in\mathbb{N}$ with $n\leq{c_{j_0-1}}$ there exists $j\in\{1,\dots,j_0-1\}$ such that $\sigma(n)\in{N_j}$ and by \eqref{equ:6:converg} we obtain
	\begin{align*}
	\epsilon>\delta&\geq|x_m(\sigma(n))-x(\sigma(n))|=|x_m(\sigma(n))-b_j|\\
	&=|(x_m\chi_{N_j})^*(n-c_{j-1})-b_j|=\abs{\left(\sum_{j=1}^{j_0-1}x_m\chi_{N_j}\right)^*(n)-b_j}{}{}
	\end{align*}
	for all $m\geq{M_{\delta}}$. On the other hand, if $n>c_{j_0-1}$ then there is $j\geq{j_0}$ such that $\sigma(n)\in N_j$ and by \eqref{equ:5:converg} and \eqref{equ:6:converg} it follows that
	\begin{align*}
	\abs{\left(\sum_{j=j_0}^{\infty}x_m\chi_{N_j}\right)^*(n-c_{j_0-1})-x^*(n)}{}{}&=|x_m(\sigma(n))\chi_{N_j}(\sigma(n))-x(\sigma(n))|\\
	&=|x_m(\sigma(n))-b_j|<\epsilon
	\end{align*}
	for all $m\geq{M_{\delta}}$. Now, let us notice that for every $n\in\mathbb{N}$,
	\begin{equation*}
	x^*(n)=\sum_{j=1}^{\infty}b_j\chi_{\{i\in\mathbb{N}:c_{j-1}+1\leq{i}\leq{c_j}\}}(n)
	\end{equation*}	
	and 
	\begin{equation*}
	x_m^*(n)=
	\begin{cases}
	\left(\sum_{j=1}^{j_0-1}x_m\chi_{N_j}\right)^*(n)&\textnormal{if}\quad{n}\leq{c_{j_0-1}},\\
	\left(\sum_{j=j_0}^{\infty}x_m\chi_{N_j}\right)^*(n-c_{j_0-1})&\textnormal{if}\quad{n}>{c_{j_0-1}}.
	\end{cases}
	\end{equation*}
	Hence, we infer that for any $m\geq{M_\delta}$ and $n\in\mathbb{N}$,
	\begin{equation*}
	x_m^*(n)\rightarrow{x^*(n)}.
	\end{equation*}
	Now, we assume that $b>0$. Then, it is easy to see that $x^*(\infty)=b>0$. Next, taking 
	\begin{equation*}
	y=x\chi_{\supp(x)}+b\chi_{\mathbb{N}\setminus\supp(x)}\quad\textnormal{and}\quad{}y_m=x_m\chi_{\supp(x)}+b\chi_{\mathbb{N}\setminus\supp(x)}
	\end{equation*}
	for all $m\in\mathbb{N}$, we may show that $x^*=y^*$ and $x_m^*=y_m^*$ for sufficiently large $m\in\mathbb{N}$. Next, passing to subsequence and relabeling if necessary, it is enough to prove that $y_m^*\rightarrow{y^*}$ on $\mathbb{N}$. Clearly, by definition of $y$ and $y_m$ for all $m\in\mathbb{N}$ we may observe that $y_m-b$ converges $y-b$ globally in measure and $(y-b)^*(\infty)=0$. Finally, using analogous technique as previously, in case $3$ for $b=0$, we finish the proof.	
\end{proof}

\begin{remark}\label{property:sequence}
	Let us notice that using analogous techniques as in the proof of the property $9^0$ in \cite{KPS} and by the property $7^0$ in \cite{KPS} (see pp. 64-65), in view of Theorem 2.7 and Proposition 3.3 in \cite{BS} we are able to show the below assertion.\\
	For any two sequences $x$ and $y$ with $x^*(\infty)=0$ and $y^*(\infty)=0$ the following conditions are equivalent.
	\begin{itemize}
		\item[$(i)$] For any $i\in\mathbb{N}$, $$(x+y)^*(i)={x^*(i)}+{y^*(i)}.$$
		\item[$(ii)$] $\sg(x(i))=\sg(y(i))$ for any $i\in\mathbb{N}$ and there exists $(E_n)_{n\in\mathbb{N}}$ a countable collection of sets such that for every $n\in\mathbb{N}$ we have $\card(E_n)=n$ and    
		$$x^{**}(n)=\frac{1}{n}\sum_{i\in{E_n}}|x(i)|\quad\textnormal{ and }\quad y^{**}(n)=\frac{1}{n}\sum_{i\in{E_n}}|y(i)|.$$
	\end{itemize} 
\end{remark}

\section{geometric structure of sequence lorentz spaces $\gamma_{p,w}$}

In this section, we discuss complete criteria for order continuity, the Fatou property, strict monotonicity and strict convexity and also extreme points of the unit ball in the sequence Lorentz space $\gamma_{p,w}$. 

\begin{theorem}\label{thm:OC:Lorentz}
	Let $w$ be a nonnegative weight sequence and $0<p<\infty$. The Lorentz space $\gamma_{p,w}$ is order continuous if and only if $W(\infty)=\infty$.
\end{theorem}

\begin{proof}
	\textit{Necessity.} Suppose that $\gamma_{p,w}$ is not order continuous. Then, there exists $(x_m)\subset{\gamma_{p,w}^+}\setminus\{0\}$ such that $x_m\downarrow{0}$ pointwise and $d=\inf_{n\in\mathbb{N}}\norm{x_m}{\gamma_{p,w}}{}>0$. Next, passing to subsequence and relabeling if necessary we may assume that $\norm{x_m}{\gamma_{p,w}}{}\downarrow{d}$. Since $W(\infty)=\infty$ we claim that $d_{x}(\lambda)<\infty$ for all $\lambda>0$ and $x\in\gamma_{p,w}$. Indeed, assuming for a contrary that there is $x\in\gamma_{p,w}$ such that $x^*(\infty)=\lim_{n\rightarrow\infty}x^*(n)>0$ we obtain $\ell^\infty\hookrightarrow{\gamma_{p,w}}$. Define $z=\chi_{\mathbb{N}}$. Then, we have $z^{**}=z\in\gamma_{p,w}$ and also $\norm{z}{\gamma_{p,w}}{}=W(\infty)=\infty$, which gives us a contradiction and proves the claim. Let $\epsilon>0$. Define two sets
	\begin{equation*}
	E_1=\{n\in\mathbb{N}:x_1(n)>\epsilon\}\qquad\textnormal{and}\qquad{E_2=\mathbb{N}}\setminus{E_1}.
	\end{equation*}
	Now, since $x_1^*(\infty)=0$ it is easy to notice $\card(E_1)=d_{x_1}(\epsilon)<\infty$ and $E_1\cap{E_2}=\emptyset$. Therefore, since $x_m\downarrow{0}$ pointwise we have
	\begin{equation*}
	d_{x_m}(\epsilon)=\card\{n\in\mathbb{N}:x_m(n)>\epsilon\}\rightarrow{0}\quad\textnormal{as}\quad m\rightarrow\infty.
	\end{equation*}
	Hence, by Lemma \ref{lem:properties} it follows that $x_m^*\rightarrow{0}$ pointwise on $\mathbb{N}$. Consequently, since $\norm{x_1}{\gamma_{p,w}}{}<\infty$ and $x^{**}(n)<\infty$ for all $n\in\mathbb{N}$, applying twice the Lebesgue Dominated Convergence Theorem we conclude $\norm{x_m}{\gamma_{p,w}}{}\rightarrow{0}$.	\\
	\textit{Sufficiency.}
	Assume for a contrary that $W(\infty)<\infty$. Then, it is easy to see that $x=\chi_{\mathbb{N}}\in\gamma_{p,w}$, $x^{**}=x$ and $\norm{x}{\gamma_{p,w}}{}=W(\infty)$. Define $x_m=\chi_{\{i\in\mathbb{N}:i\geq m\}}$ for any $m\in\mathbb{N}$. Clearly, we have $x_m\downarrow{0}$ and $x_m\leq{x}$ pointwise for every $m\in\mathbb{N}$. Moreover, we observe that $x_m^{**}=x^{**}$ for any $m\in\mathbb{N}$. Hence, we get $\norm{x_m}{\gamma_{p,w}}{}=W(\infty)>0$ for all $n\in\mathbb{N}$, which contradicts with assumption that $\gamma_{p,w}$ is order continuous.
\end{proof}

\begin{remark}\label{rem:embedding}
	First, let us observe that for any sequence symmetric space $E$, Proposition 5.9 in \cite{BS} is true. Namely, using analogous technique as in \cite{BS} we clearly get the embedding $E\hookrightarrow{m_\phi}$ holds with constant $1$, i.e. for all $x\in{E}$,
	\begin{equation*}
	\norm{x}{m_\phi}{}=\sup\{x^{**}(n)\phi_E(n)\}\leq\norm{x}{E}{},
	\end{equation*}
	where $\phi_E$ is the fundamental sequence of a symmetric space $E$ on $\mathbb{N}$. Next, in view of Remark 3.2 in \cite{Cies-FR} and assuming that $E$ has the Fatou property, we may show that $\phi_E(\infty)=\infty$ if and only if $x^*(\infty)=0$ for any $x\in{E}$.
\end{remark}

\begin{lemma}\label{lem:Fatou:Lorentz}
	Let $w$ be a nonnegative weight sequence and $0<p<\infty$. The Lorentz space $\gamma_{p,w}$ has the Fatou Property.
\end{lemma}

\begin{proof}
	Let $(x_m)\subset\gamma_{p,w}^+$, $x\in\ell^0$ and $x_m\uparrow{x}$ pointwise and $\sup_{m\in\mathbb{N}}\norm{x_m}{\gamma_{p,w}}{} <\infty$. Immediately, by Proposition 1.7 in \cite{BS} it follows that $x_m^*\uparrow{x^*}$. Next, applying twice Lebesgue Monotone Convergence Theorem \cite{Royd} we get $\norm{x_m}{\gamma_{p,w}}{}\uparrow\norm{x}{\gamma_{p,w}}{}$. Finally, since $\sup_{m\in\mathbb{N}}\norm{x_m}{\gamma_{p,w}}{}<\infty$ it follows that $x\in\gamma_{p,w}$.
\end{proof}

\begin{theorem}\label{thm:SM:Lorentz}
	Let $w$ be a nonnegative weight sequence and $0<p<\infty$. The Lorentz space $\gamma_{p,w}$ is strictly monotone if and only if $W(\infty)=\infty$.
\end{theorem}

\begin{proof}
	\textit{Necessity.}	Assume for a contrary that $W(\infty)<\infty$. Then, we may show that $\ell^\infty\hookrightarrow\gamma_{p,w}$. Next, defining two sequences 
	\begin{equation*}
	x=\chi_{\{i\in\mathbb{N}:i>1\}}\qquad\textnormal{and}\qquad{y=\chi_{\mathbb{N}}}
	\end{equation*}
	we easily observe that $x\leq{y}$, $x\neq{y}$ and $x^{**}=y^{**}=y$. Consequently, $\norm{x}{\gamma_{p,w}}{}=\norm{y}{\gamma_{p,w}}{}$, which contradicts with assumption that the Lorentz space $\gamma_{p,w}$ is strictly monotone.\\	
	\textit{Sufficiency.} Let $x,y\in\gamma_{p,w}^+$, $x\leq{y}$ and $x\neq{y}$. Since $x\neq{y}$ there exists $n_0\in\mathbb{N}$ such that $x(n_0)<y(n_0)$. 
	Define 
	\begin{equation*}
	\delta_0=\max\{y(n_0)/2,x(n_0)\}\qquad\textnormal{and}\qquad N_0=\{n\in\mathbb{N}:y(n)>\delta_0\}.
	\end{equation*} 
	Since $W(\infty)=\infty$, by the proof of Theorem \ref{thm:OC:Lorentz} it follows that $y^*(\infty)=x^*(\infty)=0$. Hence, since $n_0\in N_0$ we get 
	\begin{equation*}
	0<\card(N_0)=d_y(\delta_0)<\infty.
	\end{equation*}
	Now, we claim that there exists $m_0\in\{1,\dots,\card(N_0)\}$ such that $x^*(m_0)<y^*(m_0)$. Indeed, if it is not true then for all $n\in\{1,\dots,\card(N_0)\}$ we have $x^*(n)=y^*(n)$. Moreover, there is a permutation $\sigma:\mathbb{N}\rightarrow\mathbb{N}$ such that $\sigma(n)\in{N_0}$ and $y^*(n)=y(\sigma(n))$ for every $n\in\{1,\dots,\card(N_0)\}$. So, we have
	\begin{equation*}
	x^*(n)=y^*(n)=y(\sigma(n))\geq{x(\sigma(n))}
	\end{equation*}
	for any $n\in\{1,\dots,\card(N_0)\}$. Let $m_0\in\{1,\dots,\card(N_0)\}$ be such that $\sigma(m_0)=n_0$. Then, we observe that
	\begin{equation*}
	x^*(m_0)=y^*(m_0)=y(\sigma(m_0))=y(n_0)>{x(n_0)}.
	\end{equation*}
	Therefore, we obtain
	\begin{equation*}
	x^*(m_0)>{x(n_0)}=x(\sigma(m_0)),
	\end{equation*}
	which implies that there exists $k_0\in\mathbb{N}\setminus{N_0}$ such that $x(k_0)=y(n_0)$. On the other hand, it is well known that $x(k_0)\leq y(k_0)$, whence
	\begin{equation*}
	y(k_0)\geq x(k_0)=y(n_0)=y^*(m_0).
	\end{equation*}
	In consequence, by definition of $N_0$ this yields that $k_0\in{N_0}$ and gives us a contradiction. Now, since $x^*(n)\leq{y^*}(n)$ for any $n\in\mathbb{N}$ and $x^*(n_0)<y^*(n_0)$ for some $n_0\in\mathbb{N}$ it follows that 
	\begin{equation*}
	x^{**}(n)\leq{y^{**}(n)}\qquad\textnormal{and}\qquad\sum_{i=1}^{k}x^*(i)<\sum_{i=1}^{k}y^*(i)
	\end{equation*}
	for all $n\in\mathbb{N}$ and $k\geq{n_0}$. Finally, by assumption that $W(\infty)=\infty$ there exists $(n_k)\subset\mathbb{N}$ such that $w(n_k)>0$ for every $k\in\mathbb{N}$. Hence, we infer that $\norm{x}{\gamma_{p,w}}{}<\norm{y}{\gamma_{p,w}}{}$.	
\end{proof}

The immediate consequence of the previous theorem and Proposition 2.1 in \cite{KamLeeLew} is the following result.

\begin{corollary}\label{coro:extrem-point}
	Let $w\geq{0}$ be a weight sequence such that $W(\infty)=\infty$ and let $0<p<\infty$. An element $x\in S_{\gamma_{p,w}}$ is an extreme point of $B_{\gamma_{p,w}}$ if and only if $x^*$ is an extreme point of $B_{\gamma_{p,w}}$.
\end{corollary}

Next, we show that the Lorentz space $\gamma_{p,w}$ is strictly convex for $1<p<\infty$ and $w$ a positive weight sequence such that $W(\infty)=\infty$. In some parts of the proof of the following theorem we use the similar techniques to Theorem 3.1 in \cite{CiesKamPluc} (see also Theorem 2.3 in \cite{CKKP}). For the sake of completeness and reader's convenience we show all details of the proof.

\begin{theorem}
	Let $w$ be a nonnegative weight sequence. The Lorentz space $\gamma_{p,w}$ is strictly convex if and only if $1<p<\infty$ and $w(n)>0$ for any $n\in\mathbb{N}$ and also $W(\infty)=\infty$.  
\end{theorem}

\begin{proof}
	\textit{Necessity.} Assume that $\gamma_{p,w}$ is strictly convex. For a contrary we suppose that $p=1$. Let $x,y\in S_{\gamma_{p,w}}$ and $\norm{x+y}{\gamma_{p,w}}{}=2$. Without loss of generality we may assume that $x=x^*$ and $y=y^*$. Then, we have $(x+y)^{**}=x^{**}+y^{**}$ and also
	$$\norm{x+y}{\gamma_{p,w}}{}=\norm{x}{\gamma_{p,w}}{}+\norm{y}{\gamma_{p,w}}{}=2.$$
	Consequently, since $x$ and $y$ are arbitrary and $\gamma_{p,w}$ is strictly convex we conclude a contradiction. Now, assume that $W(\infty)<\infty$. Define 
	\begin{equation*}
	x=\frac{1}{W(\infty)^{1/p}}\chi_{\{2n:n\in\mathbb{N}\}}\quad\textnormal{ and }\quad{y}=\frac{1}{W(\infty)^{1/p}}\chi_{\mathbb{N}}.
	\end{equation*}
	Clearly, we have for any $n\in\mathbb{N}$,
	\begin{equation*}
	x^{**}(n)=y^{**}(n)=\frac{1}{W(\infty)^{1/p}}.
	\end{equation*}
	Moreover, we observe that 
	\begin{equation*}
	(x+y)^{**}(n)=\frac{1}{W(\infty)^{1/p}}\left(2\chi_{\{2n:n\in\mathbb{N}\}}+\chi_{\{2n-1:n\in\mathbb{N}\}}\right)^{**}(n)=\frac{2}{W(\infty)^{1/p}}
	\end{equation*}
	for any $n\in\mathbb{N}$. Hence, we get
	\begin{equation*}
	\norm{x}{\gamma_{p,w}}{}=\norm{y}{\gamma_{p,w}}{}=\frac{\norm{x+y}{\gamma_{p,w}}{}}{2}=1.
	\end{equation*}
	Therefore, by assumption that $\gamma_{p,w}$ is strictly convex we obtain a contradiction. Now, let us suppose for a contrary that there is $n_0\in\mathbb{N}$ such that $w(n_0)=0$. If $n_0=1$, then take $\epsilon\in(0,1/\phi_{\gamma_{p,w}}(2))$ and define 
	\begin{equation*}
	x=\frac{1}{\phi_{\gamma_{p,w}}(2)}\chi_{\{1,2\}}\qquad\textnormal{and}\qquad{}y=\left(\frac{1}{\phi_{\gamma_{p,w}}(2)}+\epsilon\right)\chi_{\{1\}}+\left(\frac{1}{\phi_{\gamma_{p,w}}(2)}-\epsilon\right)\chi_{\{2\}}.
	\end{equation*}
	It is easy to see that $x\neq{y}$ and
	\begin{equation*}
	x^{**}(n)=\frac{1}{\phi_{\gamma_{p,w}}(2)}\chi_{\{1,2\}}(n)+\frac{2}{n\phi_{\gamma_{p,w}}(2)}\chi_{\mathbb{N}\setminus\{1,2\}}(n)
	\end{equation*}
	and also
	\begin{equation*}
	y^{**}(n)=\left(\frac{1}{\phi_{\gamma_{p,w}}(2)}+\epsilon\right)\chi_{\{1\}}(n)+\frac{1}{\phi_{\gamma_{p,w}}(2)}\chi_{\{2\}}(n)+\frac{2}{n\phi_{\gamma_{p,w}}(2)}\chi_{\mathbb{N}\setminus\{1,2\}}(n).
	\end{equation*}
	Therefore, since $w(1)=0$, we have
	\begin{align*}
	\norm{x}{\gamma_{p,w}}{}=&\norm{y}{\gamma_{p,w}}{}=\left(\frac{1}{(\phi_{\gamma_{p,w}}(2))^p}w(2)+\frac{2^p}{(\phi_{\gamma_{p,w}}(2))^p}\sum_{n=3}^{\infty}\frac{w(n)}{n^p}\right)^{1/p}\\
	=&\frac{1}{\phi_{\gamma_{p,w}}(2)}\left(W(2)+W_p(2)\right)^{1/p}=1.
	\end{align*}
	Furthermore, we observe that  
	\begin{align*}
	(x+y)^{**}(n)=&\left(\left(\frac{2}{\phi_{\gamma_{p,w}}(2)}+\epsilon\right)\chi_{\{1\}}+\left(\frac{2}{\phi_{\gamma_{p,w}}(2)}-\epsilon\right)\chi_{\{2\}}\right)^{**}(n)\\
	=&\left(\frac{2}{\phi_{\gamma_{p,w}}(2)}+\epsilon\right)\chi_{\{1\}}(n)+\frac{4}{n\phi_{\gamma_{p,w}}(2)}\chi_{\mathbb{N}\setminus\{1\}}(n).
	\end{align*}
	Hence, since $w(1)=0$, we get
	\begin{align*}
	\norm{x+y}{\gamma_{p,w}}{}=&\left(\frac{4^p}{(\phi_{\gamma_{p,w}}(2))^p}\sum_{n=2}^{\infty}\frac{w(n)}{n^p}\right)^{1/p}=\frac{2}{\phi_{\gamma_{p,w}}(2)}\left(w(2)+2^p\sum_{n=3}^{\infty}\frac{w(n)}{n^p}\right)^{1/p}\\
	=&\frac{2}{\phi_{\gamma_{p,w}}(2)}\left(W(2)+W_p(2)\right)^{1/p}=2.	
	\end{align*}
	So, in case when $w(1)=0$, it follows that $\gamma_{p,w}$ is not strictly convex. Assume that $n_0>1$. Define
	\begin{equation*}
	x=\frac{1}{\phi_{\gamma_{p,w}}(n_0)}\chi_{[1,n_0]}\qquad\textnormal{and}\qquad{y}=\frac{1}{\phi_{\gamma_{p,w}}(n_0)}\left(\chi_{[1,n_0-1]}+\frac{1}{2}\chi_{\{n_0,n_0+1\}}\right).
	\end{equation*}
	Then, we easily observe that $x\neq y$ and $\norm{x}{\gamma_{p,w}}{}=1$. Moreover, we have
	\begin{equation*}
	y^{**}(n)=\frac{1}{\phi_{\gamma_{p,w}}(n_0)}
	\begin{cases}
	1&\textnormal{if }n<n_0,\\
	\frac{n_0-1/2}{n_0}&\textnormal{if } n=n_0,\\
	\frac{n_0}{n}&\textnormal{if } n>n_0,
	\end{cases}
	\end{equation*}
	and 
	\begin{align*}
	(x+y)^{**}(n)&=\frac{1}{\phi_{\gamma_{p,w}}(n_0)}\left(2\chi_{[1,n_0-1]}+\frac{3}{2}\chi_{\{n_0\}}+\frac{1}{2}\chi_{\{n_0+1\}}\right)^{**}(n)\\
	&=\frac{2}{\phi_{\gamma_{p,w}}(n_0)}
	\begin{cases}
	1&\textnormal{if }n<n_0,\\
	\frac{n_0-1/4}{n_0}&\textnormal{if } n=n_0,\\
	\frac{n_0}{n}&\textnormal{if } n>n_0.
	\end{cases}
	\end{align*}
	Hence, since $w(n_0)=0$, we conclude that
	\begin{equation*}
	\norm{y}{\gamma_{p,w}}{}=\frac{\norm{x+y}{\gamma_{p,w}}{}}{2}=\frac{1}{\phi_{\gamma_{p,w}}(n_0)}\left(W(n_0-1)+W_p(n_0)\right)^{1/p}=1.
	\end{equation*}
	In consequence, by assumption that $\gamma_{p,w}$ is strictly convex we get a contradiction.\\	 
	\textit{Sufficiency.} Let $x,y\in{S}_{\gamma_{p,w}}$ and $x\neq{y}$. We consider the proof in two cases.\\
	\textit{Case $1$.} Assume that there exists $n_0\in\mathbb{N}$ such that $x^{**}(n_0)\neq{y}^{**}(n_0)$. Then, by strict convexity of the power function $u^p$ for $1<p<\infty$ we have 
	\begin{equation*}
	\left(\frac{1}{2}x^{**}(n_0)+\frac{1}{2}y^{**}(n_0)\right)^{p}<\frac{1}{2}x^{**p}(n_0)+\frac{1}{2}{y}^{**p}(n_0).
	\end{equation*}
	Therefore, since for any $n\in\mathbb{N}$,
	\begin{equation*}
	\left(\frac{1}{2}x^{**}(n)+\frac{1}{2}y^{**}(n)\right)^{p}\leq\frac{1}{2}x^{**p}(n)+\frac{1}{2}{y}^{**p}(n)
	\end{equation*}
	by assumption that $w(n)>0$ for all $n\in\mathbb{N}$ we infer that $\norm{x+y}{\gamma_{p,w}}{}<2$.\\
	\textit{Case $2$.} Suppose that $x^{**}(n)=y^{**}(n)$ for every $n\in\mathbb{N}$. Thus, we have $x^*(n)=y^*(n)$ for any $n\in\mathbb{N}$. We claim that there exists $n_0\in\mathbb{N}$ such that 
	\begin{equation*}
	(x+y)^{**}(n_0)<x^{**}(n_0)+y^{**}(n_0).
	\end{equation*}
	Indeed, assuming that it is not true it follows that $(x+y)^{*}(n)=x^{*}(n)+y^{*}(n)$ for all $n\in\mathbb{N}$. Consequently, since $W(\infty)=\infty$, by Remark \ref{property:sequence} we obtain $|x+y|(n)=|x(n)|+|y(n)|$ for all $n\in\mathbb{N}$ and there exists $(E_n)$ an increasing sequence of sets such that $\card(E_n)=n$ for every $n\in\mathbb{N}$ and also
	\begin{equation*}
	\sum_{i\in{E_n}}|x(i)|=\sum_{i=1}^n{x^*}=\sum_{i=1}^n{y^*}=\sum_{i\in{E_n}}|y(i)|.
	\end{equation*}
	In consequence, $|x(n)|=|y(n)|$ for any $n\in\mathbb{N}$ and so $x(n)=y(n)$ for every $n\in\mathbb{N}$. Therefore, in view of assumption $x\neq{y}$ we get a contradiction. Finally, applying the triangle inequality for the maximal function we infer that 
	\begin{equation*}
	\norm{\frac{x+y}{2}}{\gamma_{p,w}}{p}<\frac{1}{2}\norm{x}{\gamma_{p,w}}{p}+\frac{1}{2}\norm{y}{\gamma_{p,w}}{p}=1.
	\end{equation*}	
\end{proof}

Finally, we present a complete criteria for an extreme point in the ball of the Lorentz space $\gamma_{1,w}$. It is worth mentioning that in some parts of the proof we use similar technique to the proof of Theorem 2.6 in \cite{KamLeeLew}. For the sake of completeness and reader's convenience we present all details of the proof of the following theorem.

\begin{theorem}\label{thm:extreme}
	Let $w\geq{0}$ be a weight sequence such that $W(\infty)=\infty$. An element $x\in S_{\gamma_{1,w}}$ is an extreme point of $B_{\gamma_{1,w}}$ if and only if there exists $n_0\in\mathbb{N}$ such that 
	\begin{equation}\label{extreme:point}
	x^*=\frac{1}{\phi_{\gamma_{1,w}}(n_0)}\chi_{\{i\in\mathbb{N}:i\leq n_0\}}
	\end{equation}
	and in case when $n_0>1$, $W(n_0-1)>0$.
\end{theorem}

\begin{proof}
	Letting $x\in{S_{\gamma_{1,w}}}$, by Corollary \ref{coro:extrem-point} we may consider that $x=x^*$ is an extreme point of $B_{\gamma_{1,w}}$. Denote
	\begin{equation*}
	n_0=\sup\{n\in\mathbb{N}:x^*(n)=x^*(1)\}.
	\end{equation*}
	Since $W(\infty)=\infty$ and $\phi_{\gamma_{1,w}}(n)=W(n)+W_1(n)$ for any $n\in\mathbb{N}$, by Lemma \ref{lem:Fatou:Lorentz} and by Remark \ref{rem:embedding} it follows that $x^*(\infty)=0$ and so $n_0\in\mathbb{N}$. We claim that $x^*(n_0+1)=0$. Suppose on the contrary that $x^*(n_0+1)>0$ and denote
	\begin{equation*}
	n_1=\card\{n\in\mathbb{N}:x^*(n)=x^*(n_0+1)\}
	\end{equation*}
	and 
	\begin{equation*}
	d=\min\{x^*(1)-x^*(n_0+1),x^*(n_0+1)-x^*(n_0+n_1+1)\}.
	\end{equation*}
	First, notice that $\phi_{\gamma_{1,w}}(n+1)>\phi_{\gamma_{1,w}}(n)>0$ for any $n\in\mathbb{N}$. Indeed, since $W(\infty)=\infty$ we infer that $\phi_{\gamma_{1,w}}(n)>0$ for all $n\in\mathbb{N}$. Now, assuming for a contrary that there is $n\in\mathbb{N}$ such that $\phi_{\gamma_{1,w}}(n+1)=\phi_{\gamma_{1,w}}(n)$, we easily obtain
	\begin{equation*}
	w(n+1)=-(n+1)\sum_{i=n+2}^{\infty}\frac{w(i)}{i}<0.
	\end{equation*}
	Hence, since $w(n+1)\geq{0}$ we get a contradiction. Now, we are able to find $a,b\in(0,d)$ such that 
	\begin{equation}\label{equ:1:thm:extreme}
	b=a\frac{\phi_{\gamma_{1,w}}(n_0+n_1)-\phi_{\gamma_{1,w}}(n_0)}{\phi_{\gamma_{1,w}}(n_0)}.
	\end{equation}
	Define
	\begin{equation*}
	y=x^*-b\chi_{\{i\in\mathbb{N}:i\leq n_0\}}+a\chi_{\{i\in\mathbb{N}:n_0<i\leq n_0+n_1\}}
	\end{equation*}
	and 
	\begin{equation*}
	z=x^*+b\chi_{\{i\in\mathbb{N}:i\leq n_0\}}-a\chi_{\{i\in\mathbb{N}:n_0<i\leq n_0+n_1\}}.
	\end{equation*}
	Clearly, $y\neq{z}$ and $x=(y+z)/2$. Since $y=y^*$ and $z=z^*$, by \eqref{equ:1:thm:extreme} we have
	\begin{align*}
	\norm{y}{\gamma_{1,w}}{}=&\sum_{n=1}^{\infty}y^{**}(n){w(n)}\\
	=&\sum_{n=1}^{\infty}\frac{w(n)}{n}\sum_{j=1}^n\left(x^{*}(j)-b\chi_{\{i\in\mathbb{N}:i\leq n_0\}}(j)+a\chi_{\{i\in\mathbb{N}:n_0<i\leq n_0+n_1\}}(j)\right)\\
	=&\sum_{n=1}^{\infty}x^{**}(n)w(n)-b\left(\sum_{n=1}^{n_0}w(n)+n_0\sum_{n=n_0+1}^{\infty}\frac{w(n)}{n}\right)\\
	&+a\left(\sum_{n=n_0+1}^{n_0+n_1}\frac{w(n)}{n}(n-n_0)+n_1\sum_{n=n_0+n_1+1}^{\infty}\frac{w(n)}{n}\right)\\
	=&\norm{x}{\gamma_{1,w}}{}-b\phi_{\gamma_{1,w}}(n_0)+a\left(\phi_{\gamma_{1,w}}(n_0+n_1)-\phi_{\gamma_{1,w}}(n_0)\right)\\
	=&\norm{x}{\gamma_{1,w}}{}=1.
	\end{align*}
	Similarly, we may show that $\norm{z}{\gamma_{1,w}}{}=1$. Therefore, in view of assumption that $x$ is an extreme point of $B_{\gamma_{1,w}}$ we conclude a contradiction, which proves our claim.
	In case when $n_0>1$ we assume that $w(n)=0$ for all $n\in\{1,\dots,n_0-1\}$. Then, for $a\in(0,x^*(n_0))$ we define 
	\begin{equation*}
	y=x^*+a\chi_{\{1\}}-a\chi_{\{n_0\}}\quad\textnormal{and}\quad{}z=x^*-a\chi_{\{1\}}+a\chi_{\{n_0\}}.
	\end{equation*}
	Next, it is clearly observe that $y\neq{z}$, $x=(y+z)/2$, $y^*=y=z^*$ and 
	\begin{equation*}
	\norm{z}{\gamma_{1,w}}{}=\norm{y}{\gamma_{1,w}}{}=\sum_{n=n_0}^{\infty}\frac{w(n)}{n}\sum_{j=1}^n\left(x^{*}(j)+a\chi_{\{1\}}(j)-a\chi_{\{n_0\}}(j)\right)=1.
	\end{equation*}
	Consequently, by assumption that $x$ is an extreme point of $B_{\gamma_{1,w}}$ we have a contradiction. So, this implies that if $n_0>1$ then it is needed $W(n_0-1)>0$. Now, assume that $x\in\gamma_{1,w}$ and satisfies \eqref{extreme:point}. For simplicity of our notation we denote $c=1/\gamma_{1,w}(n_0)$. If $n_0=1$, then by Theorem \ref{thm:SM:Lorentz} we conclude that $x$ is an extreme point of $B_{\gamma_{1,w}}$. Consider that $n_0>1$. suppose that $y,z\in{S_{\gamma_{1,w}}}$, $y\neq{z}$ and $x=(y+z)/2$. We claim that $y(i)=z(i)=0$ for all $i>n_0$. Indeed, if $y(i)>0$ for some $i>n_0$, then it is obvious that $z(i)=-y(i)<0$ for some $i>n_0$. Next, defining two elements
	\begin{equation*}
	u=y\chi_{\{i\in\mathbb{N}:i\leq n_0\}}\quad\textnormal{ and }\quad{}v=z\chi_{\{i\in\mathbb{N}:i\leq n_0\}}
	\end{equation*}
	we have $x=(u+v)/2$. On the other hand, by Theorem \ref{thm:SM:Lorentz} we infer that $\norm{u}{\gamma_{1,w}}{}<\norm{y}{\gamma_{1,w}}{}=1$ and $\norm{v}{\gamma_{1,w}}{}<\norm{z}{\gamma_{1,w}}{}=1$. In consequence, we get
	\begin{equation*}
	1=\norm{x}{\gamma_{1,w}}{}=\frac{1}{2}\norm{u+v}{\gamma_{1,w}}{}\leq\frac{\norm{u}{\gamma_{1,w}}{}+\norm{v}{\gamma_{1,w}}{}}{2}<1,
	\end{equation*}  
	which yields a contradiction and proves our claim. Now, define 
	\begin{align*}
	&I_1=\{i\in\mathbb{N},i\leq{n_0};y(i)>c\},\\
	&I_2=\{i\in\mathbb{N},i\leq{n_0};y(i)=c\},\\
	&I_3=\{i\in\mathbb{N},i\leq{n_0};y(i)<c\}.
	\end{align*}
	We can easily notice that $y,z\in\gamma_{1,w}^+$. Indeed, if it is not true then we may define $u,v\in\gamma_{1,w}^+$ such that $u\leq|y|$, $u\neq|y|$ and $v\leq|z|$, $v\neq|z|$ and also $x=(u+v)/2$. Therefore, by Theorem \ref{thm:SM:Lorentz} we obtain a contradiction. Next, since $\gamma_{1,w}$ is strictly monotone and $y\in{S_{\gamma_{1,w}}}$, $y\neq{x}$ we observe that $\card(I_1)>0$ and $\card(I_3)>0$, whence $y(1)>y(n_0)$. Without loss of generality we may assume that $y=y^*$. Then, we have
	\begin{align}\label{equ:2:thm:extreme}
	1=&\sum_{n=1}^{n_0-1}y^{**}(n)w(n)+\sum_{i=1}^{n_0}y(i)\sum_{n=n_0}^\infty\frac{w(n)}{n}\\
	=&\sum_{n=1}^{n_0-1}\sum_{i=1}^{n}y(i)\frac{w(n)}{n}+\sum_{i=1}^{n_0}y(i)\sum_{n=n_0}^\infty\frac{w(n)}{n}.\nonumber
	\end{align}
	Moreover, by assumption that $z\in{S_{\gamma_{1,w}}}$ and $x=(y+z)/2$ it follows that $z(i)=2c-y(i)$ for any $i\in\{1,\dots,n_0\}$ and $z(i)=0$ for all $i>n_0$. Thus, we obtain 
	\begin{equation*}
	z^*(n)=\left(2c-y(n_0+1-n)\right)\chi_{\{i\in\mathbb{N}:i\leq n_0\}}(n)
	\end{equation*}
	for every $n\in\mathbb{N}$. Consequently, we have
	\begin{align*}
	1=&\sum_{n=1}^{n_0}z^{**}(n)w(n)+\sum_{i=1}^{n_0}z(i)\sum_{n=n_0+1}^\infty\frac{w(n)}{n}\\
	=&\sum_{n=1}^{n_0}\left(2cn-\sum_{i=1}^{n}y(n_0+1-i)\right)\frac{w(n)}{n}+\left(2cn_0-\sum_{i=1}^{n_0}y(i)\right)\sum_{n=n_0+1}^\infty\frac{w(n)}{n}\\
	=&2c\phi_{\gamma_{1,w}}(n_0)-\sum_{n=1}^{n_0}\sum_{i=1}^{n}y(n_0+1-i)\frac{w(n)}{n}-\sum_{i=1}^{n_0}y(i)\sum_{n=n_0+1}^\infty\frac{w(n)}{n}.
	\end{align*}
	Hence, by definition of $c$ we obtain that
	\begin{equation}\label{equ:3:thm:extreme}
	1=\sum_{n=1}^{n_0-1}\sum_{i=1}^{n}y(n_0+1-i)\frac{w(n)}{n}+\sum_{i=1}^{n_0}y(i)\sum_{n=n_0}^\infty\frac{w(n)}{n}.
	\end{equation}
	Furthermore, since $y=y^*$ and $y(1)>y(n_0)$, we infer that for every $n<{n_0}$, 
	\begin{equation*}
	\sum_{i=1}^{n}y(i)>\sum_{i=1}^{n}y(n_0+1-i).
	\end{equation*}
	In consequence, since $W(n_0-1)>0$, by \eqref{equ:2:thm:extreme} and \eqref{equ:3:thm:extreme} we conclude 
	\begin{align*}
	1=&\sum_{n=1}^{n_0-1}\sum_{i=1}^{n}y(i)\frac{w(n)}{n}+\sum_{i=1}^{n_0}y(i)\sum_{n=n_0}^\infty\frac{w(n)}{n}\\
	>&\sum_{n=1}^{n_0-1}\sum_{i=1}^{n}y(n_0+1-i)\frac{w(n)}{n}+\sum_{i=1}^{n_0}y(i)\sum_{n=n_0}^\infty\frac{w(n)}{n}=1,
	\end{align*}
	which gives us a contradiction and finishes the proof.
\end{proof}

\section{dual and predual spaces of sequence lorentz spaces $\gamma_{1,w}$}

Now, we present a characterization of the dual and predual spaces of the sequence Lorentz space $\gamma_{1,w}$. 

\begin{theorem}\label{thm:dual}
	Let $w=(w(n))_{n\in\mathbb{N}}$ be a nonnegative weight sequence and let $\phi_{\gamma_{1,w}}$ be the fundamental sequence of the sequence Lorentz space $\gamma_{1,w}$. Then $W(\infty)=\infty$ if and only if every linear bounded functional $f$ on $\gamma_{1,w}$ has the form 
	\begin{equation*}
	f(x)=\sum_{n=1}^\infty{x(n)y(n)}\qquad\textnormal{for any}\quad x\in\gamma_{1,w},\quad\textnormal{and}\quad\norm{f}{\gamma_{1,w}^*}{}=\norm{y}{m_\psi}{}
	\end{equation*}
	where $y\in{m_{\psi}}$ and $\psi(n)=n/\phi_{\gamma_{1,w}}(n)$ for every $n\in\mathbb{N}$. 
\end{theorem}

\begin{proof}
	\textit{Sufficiency.} Suppose that $W(\infty)<\infty$. We claim that $\ell^\infty\hookrightarrow\gamma_{1,w}$. Indeed, taking $x=\chi_{\mathbb{N}}$ it is easy to see that $x^{**}=x$ and $\norm{x}{\gamma_{1,w}}{}=W(\infty)<\infty$, which implies our claim. Let $f\in\gamma_{1,w}^*$. Then, by assumption there exists $y\in{m_\psi}$ such that 
	\begin{equation*}
	\norm{f}{\gamma_{1,w}^*}{}=\norm{y}{m_\psi}{}\geq\frac{1}{\phi_{\gamma_{1,w}}(n)}\sum_{i=1}^{n}y^*(k)
	\end{equation*}
	for all $n\in\mathbb{N}$. Next, in view of the inequality 
	\begin{equation*}
	W(n)\leq\phi_{\gamma_{1,w}}(n)\leq{W(\infty)<\infty}
	\end{equation*}
	for every $n\in\mathbb{N}$, it follows that $\phi_{\gamma_{1,w}}(\infty)=W(\infty)<\infty$. Thus , we have 
	\begin{equation*}
	{\phi_{\gamma_{1,w}}(\infty)}\norm{f}{\gamma_{1,w}^*}{}\geq\sum_{i=1}^{\infty}y^*(k),
	\end{equation*}
	whence $y\in\ell^1$. Therefore, we observe that $m_\psi\hookrightarrow\ell^1$. Moreover, since $\gamma_{1,w}$ and $m_\psi$ are symmetric by Corollary 6.8 in \cite{BS} we conclude that $\gamma_{1,w}\hookrightarrow\ell^\infty$ and $\ell^1\hookrightarrow{m_\psi}$. Hence, since $\ell^\infty$ is the dual space of $\ell^1$ (see \cite{LinTza}) we have a contradiction.\\	
	\textit{Necessity.}
	Since $\gamma_{1,w}$ is a symmetric space, by Corollary 4.4 and Theorem 2.7 in \cite{BS} we get the associate space $\gamma_{1,w}'$ of $\gamma_{1,w}$ is a symmetric space and 
	\begin{equation*}
	\gamma_{1,w}'=\left\{y=(y(n)):\sup_{x\in B_{\gamma_{1,w}}}\left\{\sum_{n=1}^\infty{y^*(n)}x^*(n)\right\}<\infty\right\}.
	\end{equation*}	
	Next, in view of Theorem \ref{thm:OC:Lorentz} it follows that $\gamma_{1,w}$ is order continuous if and only if $W(\infty)=\infty$. Hence, by Theorem 4.1 in \cite{BS} we have $W(\infty)=\infty$ if and only if the dual space $\gamma_{1,w}^*$ and the associate space $\gamma_{1,w}'$ of the sequence Lorentz space $\gamma_{1,w}$ coincide, i.e. $\gamma_{1,w}^*=\gamma_{1,w}'$. Consequently, assuming that $\psi$ is the fundamental sequence of the dual space $\gamma_{1,w}^*$, by Remark \ref{rem:embedding} we conclude that $\gamma_{1,w}^*\hookrightarrow{m_{\psi}}$.
	Now, we prove the reverse embedding, i.e. ${m_{\psi}}\hookrightarrow\gamma_{1,w}^*$.
	First, by Theorem 5.2 in \cite{BS} we obtain $\psi(n)\phi_{\gamma_{1,w}}(n)=n$ for all $n\in\mathbb{N}$. Therefore, we have
	\begin{align}\label{equ:marcinkie}
	m_{\psi}&=\left\{x=(x(n))_{n\in\mathbb{N}}:\sup_{n\in \mathbb{N}}\{x^{**}(n)\psi(n)\}<\infty\right\}\\
	&=\left\{x=(x(n))_{n\in\mathbb{N}}:\sup_{n\in \mathbb{N}}\left\{\frac{\sum_{i=1}^{n}x^{*}(i)}{\phi_{\gamma_{1,w}}(n)}\right\}<\infty\right\}.\nonumber
	\end{align}
	We claim that for any $y\in{m_{\psi}}$ the mapping $f_y$ given by
	\begin{equation*}
	f_y(x)=\sum_{n=1}^\infty{x(n)y(n)}\quad\textnormal{for all}\quad x\in\gamma_{1,w}
	\end{equation*}
	is a linear bounded functional on $\gamma_{1,w}$.
	Indeed, taking $y\in{m_\psi}$, by the Hardy-Littlewood inequality we obtain that for any $x\in\gamma_{1,w}$,
	\begin{equation*}
	\abs{f_y(x)}{}{}\leq\sum_{n=1}^{\infty}|x(n)y(n)|\leq\sum_{n=1}^{\infty}x^*(n)y^*(n).
	\end{equation*}
	Now, for simplicity of our notation let us denote $[1,i]=\{1,2,\dots,i\}$ for any $i\in\mathbb{N}$. Next, picking $x\in\gamma_{1,w}$ with a finite measure support, without loss of generality we may assume that
	\begin{equation*}
	x^*(n)=\sum_{k=1}^{N}a_k\chi_{[1,i_k]}(n),
	\end{equation*}
	 for every $n\in\mathbb{N}$, where $(i_k)_{k=1}^N\subset\mathbb{N}$ is strictly increasing and $a_k>0$ for any $k\in\{1,\dots,N\}$. Then, we get
	 \begin{align}\label{equ1:dual}
	 \abs{f_y(x)}{}{}&\leq\sum_{n=1}^{\infty}x^*(n)y^*(n)=\sum_{k=1}^{N}a_k\sum_{n=1}^{\infty}\chi_{[1,i_k]}(n)y^*(n)\\
	 &=\sum_{k=1}^{N}a_k\sum_{n=1}^{i_k}y^*(n)\leq\sum_{k=1}^{N}a_k\phi_{\gamma_{1,w}}(i_k)\sup_{1\leq{k}\leq{N}}\left\{\frac{\sum_{n=1}^{i_k}y^*(n)}{\phi_{\gamma_{1,w}}(i_k)}\right\}\nonumber\\
	 &\leq\norm{y}{m_\psi}{}\sum_{k=1}^{N}a_k\phi_{\gamma_{1,w}}(i_k).\nonumber
     \end{align}
	 Furthermore, we observe that
	 \begin{align*}
	 \norm{x}{\gamma_{1,w}}{}&=\sum_{n=1}^\infty{x^{**}(n)w(n)}=\sum_{n=1}^\infty\frac{w(n)}{n}\sum_{j=1}^{n}\sum_{k=1}^{N}a_k\chi_{[1,i_k]}(j)\\
	 &=\sum_{n=1}^\infty\frac{w(n)}{n}\sum_{k=1}^{N}a_k\min\{n,i_k\}\\
	 &=\sum_{k=1}^{N}a_k\sum_{n=1}^\infty\frac{w(n)}{n}\min\{n,i_k\}
	 \end{align*}
	 and 
	 \begin{align*}
	 \phi_{\gamma_{1,w}}(i_k)=W(i_k)+W_1(i_k)=\sum_{n=1}^{\infty}\frac{w(n)}{n}\min\{n,i_k\}
	 \end{align*}
	 for every $k\in\{1,\dots,N\}$. Hence, by \eqref{equ1:dual} it follows that
	 \begin{equation}\label{equ2:dual}
	 \abs{f_y(x)}{}{}\leq\norm{y}{m_\psi}{}\norm{x}{\gamma_{1,w}}{}
	 \end{equation}
	 for any $x\in\gamma_{1,w}$ with a finite measure support. Finally, since $W(\infty)=\infty$, by Theorem \ref{thm:OC:Lorentz} we get $\gamma_{1,w}$ is order continuous, and so every element $x\in\gamma_{1,w}$ can be expressed as a limit of a sequence of elements in $\gamma_{1,w}$ with finite measure support. Thus, we conclude that \eqref{equ2:dual} holds for any $x\in\gamma_{1,w}$. Now, we show that there exists an isometry between ${m_\psi}$ and ${\gamma_{1,w}^*}$. Next, it is easy to see that there exists a permutation $\sigma:\mathbb{N}\rightarrow\mathbb{N}$ such that $y^*(n)=\abs{y\circ\sigma(n)}{}{}$ for any $n\in\mathbb{N}$. We present the proof in two cases.\\
	 \textit{Case $1.$} Assume that there is $n_0\in\mathbb{N}$ such that 
	 \begin{equation}\label{equ:assum:1}
	 \norm{y}{m_\psi}{}=\frac{1}{\phi_{\gamma_{1,w}}(n_0)}\sum_{i=1}^{n_0}y^*(i)
	 \end{equation}
	 Without loss of generality we may assume that $y^*(i)>0$ for all $i\in[1,n_0]$. Define 
	 \begin{equation*}
	 x(n)=
	 \begin{cases}
	 \frac{\sg(y(n))}{\phi_{\gamma_{1,w}}(n_0)}&\textnormal{if}\quad{}n\in\sigma{([1,n_0])},\\
	 0&\textnormal{otherwise.}
	 \end{cases}	
	 \end{equation*}
	 Then, we have for any $n\in\mathbb{N}$,
	 \begin{equation*}
	 x^{**}(n)=\frac{1}{n}\sum_{i=1}^{n}\frac{1}{\phi_{\gamma_{1,w}}(n_0)}\chi_{[1,n_0]}(i)=\frac{1}{\phi_{\gamma_{1,w}}(n_0)}\left(\chi_{[1,n_0]}(n)+\frac{n_0}{n}\chi_{\mathbb{N}\setminus[1,n_0]}(n)\right).
	 \end{equation*}
	 Consequently, we get
	 \begin{equation*}
	 \norm{x}{\gamma_{1,w}}{}=\sum_{n=1}^\infty{x^{**}(n)w(n)}=\frac{1}{\phi_{\gamma_{1,w}}(n_0)}\left(\sum_{n=1}^{n_0}w(n)+n_0\sum_{n=n_0+1}^{\infty}\frac{w(n)}{n}\right)=1.
	 \end{equation*}
	 Next, by \eqref{equ:assum:1} we observe that 
	 \begin{align}\label{equ:case:1}
	 f_y(x)=\sum_{n=1}^{\infty}x(n)y(n)&=\frac{1}{\phi_{\gamma_{1,w}}(n_0)}\sum_{n\in\sigma([1,n_0])}|y(n)|\\
	 &=\frac{1}{\phi_{\gamma_{1,w}}(n_0)}\sum_{n\in[1,n_0]}|y\circ\sigma(n)|\nonumber\\
	 &=\frac{1}{\phi_{\gamma_{1,w}}(n_0)}\sum_{n=1}^{n_0}y^*(n)=\norm{y}{m_\psi}{}.\nonumber
	 \end{align}
	 \textit{Case $2.$} Suppose that 
	 \begin{equation}\label{equ:assum:2}
	 \norm{y}{m_\psi}{}=\limsup_{n\rightarrow\infty}\frac{1}{\phi_{\gamma_{1,w}}(n)}\sum_{i=1}^{n}y^*(i).
	 \end{equation}
	 Then, defining a sequence $(x_m)$ by
	 \begin{equation*}
	 x_m(n)=
	 \begin{cases}
	 \frac{\sg(y(n))}{\phi_{\gamma_{1,w}}(m)}&\textnormal{if}\quad{}n\in\sigma{([1,m])},\\
	 0&\textnormal{otherwise.}
	 \end{cases}	
	 \end{equation*}
	 and proceeding analogously as in case $1$ we obtain $\norm{x_m}{\gamma_{1,w}}{}=1$ for all $m\in\mathbb{N}$. Moreover, by \eqref{equ:assum:2}, passing to subsequence and relabeling if necessary we have
	 \begin{equation}\label{equ:case:2}
	 f_y(x_m)=\frac{1}{\phi_{\gamma_{1,w}}(m)}\sum_{n\in\sigma([1,m])}|y(n)|=\frac{1}{\phi_{\gamma_{1,w}}(m)}\sum_{n=1}^{m}y^*(n)\rightarrow\norm{y}{m_\psi}{}.
	 \end{equation}	 
	 Finally, combining both cases and according to \eqref{equ2:dual}, \eqref{equ:case:1} and \eqref{equ:case:2} we finish the proof.
\end{proof}

\begin{theorem}\label{thm:predual}
	 Let $w$ be a nonnegative weight sequence. The Marcinkiewicz space $m_\psi^0$ is the predual of the sequence Lorentz space $\gamma_{1,w}$ if and only if $W(\infty)=\infty$, where $$\psi(n)=\frac{n}{\phi_{\gamma_{1,w}}(n)}\quad\textnormal{for any}\quad n\in\mathbb{N}.$$
	 Additionally, if $W(\infty)=\infty$ then there exists an isometry between the sequence Lorentz space $\gamma_{1,w}$ and the dual space $(m_\psi^0)^*$ of the Marcinkiewicz space $m_\psi^0$. 
\end{theorem}

\begin{proof}
	 First, we define for any $i\in\mathbb{N}$,
	 \begin{equation*}
	 v(i)=\sum_{k=i}^{\infty}\frac{w(k)}{k}.
	 \end{equation*}
	 Clearly, $(v(i))_{i\in\mathbb{N}}$ is a decreasing sequence and $0\leq v(i)<\infty$ for all $i\in\mathbb{N}$. Moreover, we easily observe that for every $i\in\mathbb{N}$,
	 \begin{equation*}
	 v(i)=\sum_{k=i}^{\infty}\frac{w(k)}{k}=w(i)+\sum_{k=i+1}^{\infty}\frac{w(k)}{k}-\frac{(i-1)w(i)}{i}.
	 \end{equation*}
	 Hence, we evaluate 
	 \begin{align*}
	 \phi_{\gamma_{1,w}}(n)&=\sum_{i=1}^n{w(i)}+n\sum_{i=n+1}^\infty\frac{w(i)}{i}\\
	 &=\sum_{i=1}^n\left(w(i)+\sum_{k=i+1}^{\infty}\frac{w(k)}{k}-\frac{(i-1)w(i)}{i}\right)=\sum_{i=1}^n{v(i)}
	 \end{align*}
	 for all $n\in\mathbb{N}$. Next, since $\psi(n)=n/\phi_{\gamma_{1,w}}(n)$ for each $n\in\mathbb{N}$, by \eqref{equ:marcinkie} and by Theorem 3.4 in \cite{KamLee} it follows that $m_\psi$ is the bidual of $m_\psi^0$ if and only if $\phi_{\gamma_{1,w}}(\infty)=\infty$. Now, we claim that $W(\infty)=\infty$ if and only if $\phi_{\gamma_{1,w}}(\infty)=\infty$. Indeed, it is easy to see that for any $n\in\mathbb{N}$,
	 \begin{align*}
	 W(n)\leq\phi_{\gamma_{1,w}}(n)\leq\sum_{i=1}^nw(i)+n\sum_{i=n+1}^\infty\frac{w(i)}{n}=W(\infty),
	 \end{align*}
	 which implies our claim. Therefore, according to Theorem \ref{thm:dual} we obtain that the Marcinkiewicz space $m_\psi^0$ is the predual of the sequence Lorentz space $\gamma_{1,w}$ if and only if $W(\infty)=\infty$. Now, we show that there exists an isometry between $(m_\psi^0)^*$ and $\gamma_{1,w}$. First, since $\phi_{\gamma_{1,w}}(\infty)=\infty$, in view of Theorem 3.2 in \cite{KamLee} it follows that $m_\psi^0$ is a non-trivial subspace of all order continuous elements of $m_\psi$. Then, defining for any $x\in\gamma_{1,w}$ the linear mapping $f_x$ by
	 \begin{equation*}
	 f_x(y)=\sum_{n=1}^\infty{x(n)y(n)}\quad\textnormal{ for any }\quad y\in{}m_\psi^0,
	 \end{equation*}
	 and proceeding analogously as in Theorem \ref{thm:dual} we are able to show that
	 \begin{equation}\label{equ:dual:2:inequal}
	 \abs{f_x(y)}{}{}\leq\norm{y}{m_\psi}{}\norm{x}{\gamma_{1,w}}{}
	 \end{equation}
	 for any $y\in{m_\psi^0}$. On the other hand, it is well known that there exists $\sigma:\mathbb{N}\rightarrow\mathbb{N}$ a permutation such that $x^*(n)=|x\circ\sigma(n)|$ for all $n\in\mathbb{N}$. Define
	 \begin{equation*}
	 y(n)=
	 \begin{cases}
	 \sg(x(n))v(\sigma^{-1}(n))&\textnormal{if }n\in\sigma(\mathbb{N}),\\
	 0&\textnormal{otherwise,}
	 \end{cases}
	 \end{equation*} 
	 for any $n\in\mathbb{N}$. Then, we have 
	 \begin{equation*}
	 f_x(y)=\sum_{n=1}^\infty{x(n)y(n)}=\sum_{n\in\sigma(N)}{|x(n)|v\circ\sigma^{-1}(n)}=\sum_{n=1}^\infty{x^*(n)v(n)}=\norm{x}{\gamma_{1,w}}{},
	 \end{equation*}	
	 whence, according to \eqref{equ:dual:2:inequal} we finish the proof.
\end{proof}

\section{Application}

This section is devoted to a relationship between the existence set and one-complemented subspaces of the sequence Lorentz space $\gamma_{1,w}$. Moreover, we present a complete characterization of smooth points in the sequence Lorentz space $\gamma_{1,w}$ and its dual space and predual space. Finally we show full criteria for extreme points in the dual space of the sequence Lorentz space $\gamma_{1,w}$.

First, let us recall some basic definitions and notations that corresponds to the best approximation. 
Let $X$ be a Banach space and $C\subset{X}$ be a nonempty set. A continuous surjective mapping $P:X\rightarrow{C}$ is called a projection onto $C$, whenever $P|_{C}=\id$, i.e. $P^2=P$. Given a subspace $V$ of a Banach space $X$, by $P(X,V)$ we denote the set of all linear bounded projections from $X$ onto $V$. Let us recall that a closed subspace $V$ of a Banach space $X$ is said to be \textit{one-complemented} if there exists a norm one projection $P\in P(X,V)$.
A set $C\subset X$ is said to be an \textit{existence set} of the best approximation if for any $x\in{X}$ we have 
\begin{equation*}
R_C(x)=\left\{y\in C:\norm{x-y}{X}{}=\inf_{c\in C}\norm{x-c}{X}{}\right\}\neq\emptyset.
\end{equation*} 
It is obvious that any one-complemented subspace is an existence set. The converse in general is not true. By a deep result of Lindenstrauss \cite{Lind} there exists a Banach space $X$ and a linear subspace $V$ of $X$ such that $V$ is an existence set in $X$ and $V$ is not one-complemented in $X$. However, if $X$ is a smooth Banach space both notions are equivalent (see Proposition 5 in \cite{BeaMor}). We will show that both notions are equivalent in $\gamma_{1,w}$, which is obviously not a smooth space.  

First, we establish an isometry between the sequence Lorentz spaces $\gamma_{1,w}$ and $d_{1,v}$ for some nonnegative sequences $w$ and $v$.

\begin{remark}\label{rem:isometry:lorentz}
	Assuming that $w=(w(n))_{n\in\mathbb{N}}$ is a nonnegative weight sequence, we may easily show that there exits a linear surjective isometry $T$ form the sequence Lorentz space $\gamma_{1,w}$ onto the sequence Lorentz space $d_{1,v}$, where $v=(v(n))_{n\in\mathbb{N}}$ is given by
	\begin{equation}\label{new:weight}
	v(i)=\sum_{n=i}^\infty\frac{w(n)}{n}\quad\textnormal{ for any }\quad i\in\mathbb{N}.
	\end{equation}
	Indeed, taking $x\in\gamma_{1,w}$ we observe that
	\begin{equation*}
	\norm{x}{\gamma_{1,w}}{}=\sum_{n=1}^\infty\frac{w(n)}{n}\sum_{i=1}^{n}{x}^{*}(i)=\sum_{i=1}^\infty{x^*(i)}\sum_{n=i}^{\infty}\frac{w(n)}{n}=\sum_{n=1}^\infty{x^*(n)v(n)}=\norm{x}{d_{1,v}}{}.
	\end{equation*}
\end{remark}

\begin{theorem}
	Let $w$ be a nonnegative weight sequence and let $V\subset\gamma_{1,w}$, $V\neq\{0\}$ be a linear subspace. If $V$ is an existence set, then $V$ is one-complemented.
\end{theorem}

\begin{proof}
	Let $v$ be a nonnegative sequence given by \eqref{new:weight}. Then, by Remark \ref{rem:isometry:lorentz} there exists a linear surjective isometry $T:\gamma_{1,w}\rightarrow{d_{1,v}}$. Hence, since $V$ is an existence set in $\gamma_{1,w}$, by Lemma 3.4 in \cite{KamLeeLew} it follows that $T(V)\neq\{0\}$ is an existence set in $d_{1,v}$. In consequence, by Theorem 3.10 in \cite{KamLeeLew} we infer that $T(V)$ is one complemented in $d_{1,v}$. Finally, applying again Lemma 3.4 in \cite{KamLeeLew} we get that $V$ is one-complemented in $\gamma_{1,w}$.
\end{proof}

We present a full criteria for smooth points in the sequence Lorentz space $\gamma_{1,w}$ and its dual and predual spaces.  
First, let us notice that by Theorem 1.10 in \cite{KamLeeLew} and by Remark \ref{rem:isometry:lorentz}, the next theorem follows immediately. 

\begin{theorem}
	Let $w$ be a nonnegative weight sequence and let $x\in{S_{\gamma_{1,w}}}$. Then, an element $x$ is a smooth point in $\gamma_{1,w}$ if and only if the following conditions are satisfied
	\begin{itemize}
		\item[$(i)$] $\card(\supp(x))=\infty$.
		\item[$(ii)$] If there is $n\in\mathbb{N}$ such that $w(n)>0$, then $x^*(n)>x^*(n+1)$.
	\end{itemize} 
\end{theorem}

\begin{theorem}
	Let $w$ be a nonnegative weight sequence and $\psi(n)=n/\phi_{\gamma_{1,w}}(n)$ for any $n\in\mathbb{N}$ and $x\in{S_{m_\psi^0}}$. Then, an element $x$ is a smooth point in $m_\psi^0$ if and only if 
	\begin{equation*}
	\card\left\{n\in\mathbb{N}:x^{**}(n)\psi(n)=1\right\}=1.
	\end{equation*}
\end{theorem}

\begin{proof}
	Let $v$ be a sequence given by \eqref{new:weight} and let $V(n)=\sum_{i=1}^n{v(i)}$. Then, by Remark \ref{rem:isometry:lorentz} we easily observe that $V(n)=\phi_{\gamma_{1,w}}(n)=\frac{n}{\psi(n)}$ for every $n\in\mathbb{N}$ and
	\begin{equation*}
	m_\psi^0=\left\{x\in{m_\psi}:\lim_{n\rightarrow\infty}\frac{1}{V(n)}\sum_{i=1}^{n}x^*(i)=0\right\}.
	\end{equation*}
	Hence, in view of Theorem 1.5 in \cite{KamLeeLew} we complete the proof.
\end{proof}

Directly, by Theorem 1.9 in \cite{KamLeeLew} and Remark \ref{rem:isometry:lorentz} and also Theorem \ref{thm:dual} we infer the following theorem.

\begin{theorem}
	Let $w$ be a nonnegative weight sequence and $\psi(n)=n/\phi_{\gamma_{1,w}}(n)$ for any $n\in\mathbb{N}$ and $x\in{S_{\gamma_{1,w}^*}}$. Then, an element $x$ is a smooth point in $B_{\gamma_{1,w}^*}$ if and only if there exists $n_0\in\mathbb{N}$ such that 
	\begin{equation*}
	x^{**}(n_0)\psi(n_0)=1>\sup_{n\neq{n_0}}\{x^{**}(n)\psi(n)\}.
	\end{equation*}
\end{theorem}

The last essential application of Theorem \ref{thm:dual} and Remark \ref{rem:isometry:lorentz}, in view of Theorem 2.2 in \cite{KamLeeLew}, is the next result which presents an equivalent condition for extreme points in the dual space $\gamma_{1,w}^*$ of the sequence Lorentz space $\gamma_{1,w}$. 

\begin{theorem}
	Let $w$ be a nonnegative weight sequence and $\psi(n)=n/\phi_{\gamma_{1,w}}(n)$ for any $n\in\mathbb{N}$ and $x\in{S_{\gamma_{1,w}^*}}$. Then, $x$ is an extreme point of $B_{\gamma_{1,w}^*}$ if and only if $x^*(n)=\sum_{i=n}^\infty\frac{w(i)}{i}$ for all $n\in\mathbb{N}$.
\end{theorem}

\begin{remark}
	 Although applying Theorem 2.6 in \cite{KamLeeLew} and Remark \ref{rem:isometry:lorentz} we are able to find successfully an equivalent condition for an extreme point in the sequence Lorentz space $\gamma_{1,w}$, with $w$ a nonnegative weight sequence, we present the proof of this problem with all details (see Theorem \ref{thm:extreme}). It is worth mentioning that the techniques, that was presented in the proof of Theorem \ref{thm:extreme}, might be interesting for readers and applicable to search a complete characteristic of an extreme point in $\gamma_{p,w}$ with $1<p<\infty$.  
\end{remark}

\subsection*{Acknowledgement}

$^*$The first author (Maciej Ciesielski) is supported by the Ministry of Science and Higher Education of Poland, grant number 04/43/DSPB/0094.

$\begin{array}{lr}
\textnormal{\small Maciej Ciesielski} & \textnormal{\small Grzegorz Lewicki}\\
\textnormal{\small Institute of Mathematics} & \textnormal{\small Department of Mathematics}\\
\textnormal{\small Faculty of Electrical Engineering} & \textnormal{\small and Computer Science}\\
\textnormal{\small Pozna\'{n} University of Technology} & \textnormal{\small Jagiellonian University}\\
\textnormal{\small Piotrowo 3A, 60-965 Pozna\'{n}, Poland} & \textnormal{\small\qquad \L ojasiewicza 6, 30-348 Krak\'ow, Poland}\\
\textnormal{\small email: maciej.ciesielski@put.poznan.pl} & \textnormal{\small email: grzegorz.lewicki@im.uj.edu.pl}
\end{array}$

\end{document}